\documentclass{amsart}
\usepackage{latexsym,amsmath,amssymb}
\usepackage{cite}
\newtheorem{thm}{Theorem}[section]
\newtheorem{cor}[thm]{Corollary}
\newtheorem{exam}[thm]{Example}

\theoremstyle{definition}\newtheorem{defn}[thm]{Definition}
\theoremstyle{remark}

\numberwithin{equation}{section}
\newcommand{\norm}[1]{\left\Vert#1\right\Vert}

\newcommand{\A}{\mathcal{A}}

\begin{document}

\title[Nuclear WCE operators]{Nuclear weighted conditional expectation operators }

\author{\sc\bf A. Ommi and Y. Estaremi}
\address{\sc }
\address{}
\email{}
\address{\sc Y. Estaremi}
\email{y.estaremi@gu.ac.ir}
\address{Department of Mathematics, Faculty of Sciences, Golestan University, Gorgan, Iran.}
\address{\sc}
\email{}
\address{}

\thanks{}

\thanks{}

\subjclass[]{}

\keywords{}

\date{}

\dedicatory{}

\commby{}

\begin{abstract}
We provide a characterisations of nuclear weighted conditional expectation operators between different $L^p(\mu)$-spaces. As a consequence, when the underlying measure space is non-atomic, the only nuclear weighted conditional expectation operator between different $L^p(\mu)$-spaces is the zero operator.
\end{abstract}

\maketitle

\section{ \sc\bf Introduction and Preliminaries}
Let $(X, \Sigma, \mu)$ be a complete $\sigma$-finite measure space. All sets and functions statements  are to be interpreted as holding up to sets of measure zero.
For a $\sigma$-subalgebra $\mathcal{A}$ of $\Sigma$, the conditional expectation operator associated with $\mathcal{A}$ is the mapping $f\rightarrow E^{\mathcal{A}}f,$ defined for all non-negative $f$ as well as for all $f\in L^p(\mathcal{F})=L^2(X, \mathcal{F}, \mu)$, where $E^{\mathcal{A}}f$ is the unique $\mathcal{A}$-measurable function satisfying

$$\int_{A}(E^{\mathcal{A}}f)d\mu=\int_{A}fd\mu \ \ \ \ \ \ \ \forall A\in \mathcal{A}.$$
We will often write $E$ for $E^{\mathcal{A}}$. This operator will play a major role in our
work. 
We now list fundamental prperties of conditional expectation operator $E$ when considered on the space $L^p(\mu)$. For any $f, g \in L^p(\mu)$ the following holds :
\begin{enumerate}
	\item If $g$ is $\A$-measurable, then $E(fg) = E(f)g$.
	\item $|E(f)|^p \leq E(|f|^p)$.
	\item If $f \geq 0$, then $E(f) \geq 0$; if $f > 0$, then $E(f) > 0$.
	\item $|E(fg)| \leq E(|f|^p)^{1/p} E(|g|^{p'})^{1/p'}$, where $\frac{1}{p} + \frac{1}{p'} = 1$ (Hölder inequality).
	\item For each $f\geq 0$, $S(E(f))$ is the smallest $\mathcal{A}$-set containing $S(f)$, where $S(f)=\{x\in X: f(x)\neq 0\}$.
\end{enumerate}
For a comprehensive treatment of these properties, we refer the reader to \cite{Rao1993}.

\begin{defn}
	 An element $A \in \A$ with $\mu(A) > 0$ is called an $\A$-atom of the measure $\mu$ if, for each $F \in \A$ with $F \subseteq A$, either $\mu(F) = 0$ or $\mu(F) = \mu(A)$. A measure space $(X, \Sigma, \mu)$ that contains no atoms is said to be non-atomic measure space.
\end{defn}
	  It is a standard fact that every $\sigma$-finite measure space $(X, \A, \mu|_{\A})$ can be  uniquely written as a disjoint union
\[
X = \left( \bigcup_{i \in \mathbb N} A_i \right) \cup B,
\]
where each $A_i$  is an atom of $\mu$ and $B$ is non-atomic measurable set(i.e, $B$ contains no atom of positive measure) (see \cite{Zaanen1967}).

More over if the $\sigma$-subalgebra $\A$ coinsides with the $\sigma$-algebra generated by the sequence of measurable sets $\{A_i\}_{i \in \mathbb N}$ of $X$ with positive measures, then for every $f \in \mathcal{D}(E)$,the conditional expection admits the explicit representation
\[
E(f) = \sum_{i=1}^{\infty} \left( \frac{1}{\mu(A_i)} \int_{A_i} f \, d\mu \right) \cdot \chi_{A_i}.
\]

\begin{defn}
[Weighted Conditional Expectation (WCE) Operator]
Let $(X, \Sigma, \mu)$ be a $\sigma$-finite measure space and let $\A$ be a $\sigma$-subalgebra of $\Sigma$, such that $(X, \A, \mu)$ is also $\sigma$-finite. The corresponding conditional expectation operator relative to $\A$ is defined by $E:= E^{\A}(f) : L^p(\mu) \to L^p(\A)$ where $1 \leq p \leq \infty$. Given measurable functions $w, u \in L^0(\mu)$ , assume that for every  $f \in L^p(\mu)$ the product $uf$ is conditionable. The weighted conditional expectation (WCE) operator with the weights $w$ and $u$ is the linear map $T =M_w E M_u : L^p(\mu) \to L^0(\mu)$ defined by $T(f) = w\cdot E(uf)$.
\end{defn}
Our interest in operators of the form of WCE operators stems from the fact that such products tend to appear often
in the study of those operators related to conditional expectation. WCE operators appear in \cite{dou}, where it is shown that every contractive projection on certain $L^1$-spaces can be
decomposed into an operator of the form $M_wEM_u$ and a nilpotent operator. In \cite{gdp,gdp1,Herron2011} operators that are
representable as products involving multiplications and conditional expectations are studied. Also, in [9],
S.T.C. Moy has characterized all operators on $L^p$-spaces of the form $EM_u$ and $M_wEM_u$. Some classical properties
of the operator $EM_u$ on $L^p$-spaces are characterized in \cite{Herron2011} and [8]. The authors have characterized
boundedness of $T$ between two different $L^p$ spaces, polar decomposition and some other classical properties of $T$
on $L^2(\Sigma)$ in \cite{Estaremi2013}. Compactness of WCE operators between different $L^p$ spaces is characterized in \cite{Estaremi2014}. In this paper we investigate the class of Nuclear WCE operators from $L^p$ into  $L^q$-spaces in cases $p>q$ and $p<q$.

\section{ \sc\bf Main Results}
In this section first we recall the definition of nuclear and absolutely summing operators. Also, we recall some facts about nuclear operators that will be used in our main results.\\

Let $X$ and $Y$ be Banach spaces and $X^*$ be the dual of $X$.
A bounded operator $T:X\to Y$ is \emph{nuclear} if it admits a representation
\begin{equation}\label{A}
Tx=\sum_{n=1}^{\infty} f_n(x)\,y_n,
\qquad x\in X,
\end{equation}
for some $(f_n)\subset X^*$ and $(y_n)\subset Y$ with
$\sum_{n=1}^{\infty}\|f_n\|_{X^*}\,\|y_n\|_{Y}<\infty$.
The \emph{nuclear norm} is
\[
\|T\|_{N}
:=\inf\left\{\sum_{n=1}^{\infty}\|f_n\|_{X^*}\,\|y_n\|_{Y}:\ T=\sum_{n=1}^{\infty} f_n\otimes y_n\right\},
\]
where $(f\otimes y)(x):=f(x)\,y$. We refer to \cite{Pietsch1972,Diestel1995} for background on nuclear and summing operators,
including the ideal properties and standard consequences used implicitly below.

\medskip
An operator $T:X\to Y$ is \emph{absolutely $1$--summing} if there exists $M>0$ such that for every finite family
$x_1,\dots,x_m\in X$,
\begin{equation}\label{eq:1summing}
\sum_{i=1}^{m}\|T(x_i)\|
\le M\,\sup_{f\in X^*,\,\|f\|\le 1}\sum_{i=1}^{m}|f(x_i)|.
\end{equation}
The least such $M$ is denoted by $\pi_1(T)$. Using the weak $\ell_1$--norm identification, \eqref{eq:1summing} is equivalent to
\[
\sum_{i=1}^{m}\|T(x_i)\|
\le M\,\sup_{|\alpha_i|=1}\left\|\sum_{i=1}^{m}\alpha_i x_i\right\|.
\]

In the following we recall the Pietsch's domination theorem that will be used in the proof of our main results.

\begin{thm}[Pietsch domination theorem {\cite[Theorem~2.12]{Diestel1995}}]\label{thm:Pietsch}
Let $1\le p<\infty$ and let $T:X\to Y$ be absolutely $p$--summing.
Then there exist $C>0$ and a Borel probability measure $\mu$ on the weak$^{*}$--compact unit ball $B_{X^*}$ such that
\[
\|Tx\|
\le C\left(\int_{B_{X^*}} |x^*(x)|^p\,d\mu(x^*)\right)^{1/p},
\qquad x\in X.
\]
In particular, for $p=1$,
\[
\|Tx\|\le C\int_{B_{X^*}} |x^*(x)|\,d\mu(x^*).
\]
\end{thm}
Here we recall the characterizations of compactness of WCE operators between different $L^p$-spaces \cite{Estaremi2014}.
\begin{thm}
Let $T=M_wEM_u: L^p(\Sigma) \to L^q(\Sigma)$. Then the followings hold:
\begin{itemize}
\item If $1< q< p< \infty$ and let $p',q'$ be conjugate component to \(p\) and \(q\), respectively, then the WCE operator $T$ is compact if and only if
\begin{enumerate}
	\item \((E|w|^{q})^{1 / q}(E|u|^{p'})^{1 / p'} = 0\) a.e. on $B$
	\item \(\sum_{i\in \mathbb{N}}(E(|w|^q))^{p'q' / (q' - p')}(A_i)(E(|u|^p))^{q' / (q'- p')}(A_i)\mu (A_i)< \infty\).
\end{enumerate}
\item If \(1 < p < q < \infty\), then $T$ is compact if and only if
\begin{enumerate}
	\item $(E(|u|^{p'})^{1 / p'}(E(|w|^q))^{1 / q} = 0)$ a.e. on $B$ 
	\item \(\lim_{i\to \infty}\frac{E(|u|^{p'})(A_i)(E(|w|^q))^{p' / q}(A_i)}{\mu(A_i)^{(p' - q') / q'}} = 0\), when the number of \(\mathcal{A}\)-atoms is not finite.
\end{enumerate}
 \end{itemize}
 \end{thm}
 \begin{thm}
Let $T=M_wEM_u: L^p(\Sigma) \to L^1(\Sigma)$. Then the followings hold:
\begin{itemize}
\item 
For $1<p<\infty$, the WCE operator $T: L^p(\Sigma) \to L^{1}(\Sigma)$ is compact if and only if \(\sum_{i\in \mathbb{N}}E(|u|^{p'})(A_{i})(E(|w|))^{p'}(A_{i})\mu (A_{i})< \infty\) and \(E(|u|^{p'})^{1 / p'}E(|w|) = 0\) a.e. on \(B\).

\item Let \((X,A,\mu)\) be a non-atomic measure space. Since for each \(\alpha >0\) and \(\beta >0\), \(\sigma (w) = \sigma (|w|^{\alpha})\subseteq \sigma (E(|w|^{\alpha})) = \sigma ((E(|w|^{\alpha}))^{\beta})\), then by the previous results, the WCE operator $T = M_w E M_u : L^p(\Sigma) \to L^q(\Sigma)$ with \(1< p< \infty\) and \(1\leq q< \infty\) is compact if and only if it is a zero operator.
\end{itemize}
\end{thm} 

\begin{thm}
Let $T=M_wEM_u: L^1(\Sigma) \to L^q(\Sigma)$. Then the followings hold:
 Let \(1< q< \infty\) and let \(X = (\bigcup_{i\in \mathbb{N}}C_{i})\cup C\), where \(\{C_{i}\}_{i\in \mathbb{N}}\) is a countable collection of pairwise disjoint \(\Sigma\)-atoms and \(C\in \Sigma\), being disjoint from each \(C_{i}\), is non-atomic. Then:
\begin{itemize}
\item If the WCE operator $T = M_w E M_u : L^{1}(\Sigma) \to L^{q}(\Sigma)$ is compact, then
	\begin{enumerate}
		\item \(E(|u|^{q'})^{1 / q'}(E(|w|^q))^{1 / q} = 0\) a.e. on $B$
		\item \(\lim_{i\to \infty}\frac{E(|u|^{q'})(A_{i})(E(|w|^q))^{q' / q}(A_{i})}{\mu(A_{i})} = 0\), when the number of \(\mathcal{A}\)-atoms is not finite.
	\end{enumerate}
\item The WCE operator $T = M_w E M_u : L^{1}(\Sigma) \to L^{q}(\Sigma)$ is compact, if the following conditions hold:
	\begin{enumerate}
		\item \(u(E(|w|^q))^{1 / q} = 0\) a.e. on $C$, and
		\item \(\lim_{i\to \infty}\frac{E(|u|^{q'})(C_{i})(E(|w|^q))^{q' / q}(C_{i})}{\mu(C_{i})} = 0\), when the number of \(\Sigma\)-atoms is not finite.
	\end{enumerate}
\end{itemize}
\end{thm}

\begin{cor}\label{cor2.5}
	Let $1 \leq p,q < \infty$ and let the measure space $(X, \A, \mu)$ be a non-atomic. Then the WCE operator $T = M_w E M_u : L^p(\Sigma) \to L^q(\Sigma)$ is nuclear if and only if $T = 0$.
\end{cor}

\begin{proof}
	Suppose $T = M_w E M_u$ is nuclear WCE operator. Since every nuclear operator is compact, $T$ is compact .  Then by Remark 2.3 of \cite{Estaremi2014} we have $T = 0$.
	
	The converse is obvious.
\end{proof}
\begin{cor}
	Let $(X, \A, \mu)$ be a measure space with the finite subset property that is not purely atomic and let $1 \leq p,q < \infty$.If the restriction of the WCE operator $T = M_w E M_u : L^p(\Sigma) \to L^q(\Sigma)$  to the non-atomic part of the measure space is non-zero operator, then $T$ can not be nuclear.
\end{cor}

\begin{proof}
	By contrary suppose that T be nuclear.Then T is compact and by \cite{Estaremi2014} we have 
	\item $k$:=\((E|w|^{q})^{1 / q}(E|u|^{p^{\prime}})^{1 / p^{\prime}} = 0\) a.e. on \(B\) where \(1\leq q< p< \infty\)
	\item $h$:=\(E(|u|^{p'})^{1 / p'}(E(|w|^q))^{1 / q} = 0\) a.e. on \(B\) where \(1\leq p< q< \infty\)
	
	Since the measure space $(X, \A, \mu)$ is not purely atomic, there exists at least one non-atomic subset $D$ such that $\mu(D) > 0 $. By hypothesis, the operator $T$ is nonzero on $D$. Consequently, the functions $k$ and $h$  must be strictly positive on a subset $D$ with positive measure (contained in $B$) which is contradiction. Therefore, $T$ can not be compact, and Hence, $T$ can not be nuclear.
\end{proof}
\begin{thm}\label{thm2.7}
	Let $1 \leq q<p < \infty$ and let \(p^{\prime}\) and \(q^{\prime}\) be conjugate components to \(p\) and \(q\) respectively. Let $\A \subseteq \Sigma$ be a $\sigma$-finite subalgebra such that the measure space $(X, \A, \mu|_{\A})$ is purely atomic. Then the bounded WCE operator $T = M_w E M_u : L^p(\Sigma) \to L^q(\Sigma)$ is nuclear if and only if
	\[
	\sum_{i=1}^{\infty} (E(|u|^{p'})(A_i))^{1/{p'}} (E(|w|^q)(A_i))^{1/q} (\mu(A_i))^{1/r} < \infty,
	\]
	where $\{A_i\}_{i=1}^{\infty}$ is a collection of pairwise disjoint $\A$-atoms satisfying $X = \bigcup_{i=1}^{\infty} A_i$ , and the exponent $r$ given by $\frac{1}{p} + \frac{1}{r} = \frac{1}{q}$.
\end{thm}

\begin{proof}
	 It is standard fact that every $\A$-measurable function on $X$ is constant on each $\A$-atom. Hence for every $f \in L^p(\Sigma)$ we have
	\begin{align*}
		T f &= w \cdot E(uf) = w \cdot \sum_{i=1}^{\infty} E(uf)(A_i) \cdot \chi_{A_i} \\
		&= w \cdot \sum_{i=1}^{\infty} \left( \frac{1}{\mu(A_i)} \int_{A_i} E(uf) \, d\mu \right) \cdot \chi_{A_i} \\
		&= w \cdot \sum_{i=1}^{\infty} \left( \frac{1}{\mu(A_i)} \int_{A_i} u f \, d\mu \right) \cdot \chi_{A_i}.	
	\end{align*}

	Assume that
	\[
	 \sum_{i=1}^{\infty} (E(|u|^{p'})(A_i))^{1/{p'}} (E(|w|^q)(A_i))^{1/q} (\mu(A_i))^{1/r} < \infty.
	\]
	Define for each $i \in \mathbb N$ 
	\[
	\phi_i(f) = \int_X (u \cdot \chi_{A_i}) f \, d\mu , g_i = \frac{w \cdot \chi_{A_i}}{\mu(A_i)}.
	\]
	Then $u \cdot \chi_{A_i} \in L^{p'}(\Sigma)$ and $g_i \in L^q(\Sigma)$. Clearly, each $\phi_i$ is a bounded linear functional on $L^p(\Sigma)$. More over one can verify that \[
	\norm{g_i}_q = (E(|w|^q)(A_i))^{1/q}(\mu(A_i))^{({1/q})-1}
	 \]
	, in which $\frac{1}{q} + \frac{1}{q'} = {1} $ (we assume that $\frac{1}{\infty} = {0} $ ) and \[
	\norm{\phi_i}_{p'} = (E(|u|^{p'})(A_i))^{1/{p'}}.
	\]
	 Therefore,
	\[
	T = M_w E M_u = \sum_{i=1}^{\infty} \phi_i \otimes g_i,
	\]
	and
	\[
	\sum_{i=1}^{\infty} \norm{\phi_i}_{p'} \norm{g_i}_{q} = \sum_{i=1}^{\infty} (E(|u|^{p'})(A_i))^{1/{p'}} (E(|w|^q)(A_i))^{1/q} (\mu(A_i))^{1/r} < \infty.
	\]
	This implies that $T$ is nuclear.
	
	Conversely, let the WCE operator $T = M_w E M_u$ be nuclear. Then $T$ is absolutely summing and so by Pietsch's theorem 2.12 of \cite{Diestel1995}, there is a Borel probability measure $\nu$ on the unit ball $B_{p'}$ of $L^{p'}(\mu) \simeq (L^p(\mu))^*$ such that for some $M > 0$,
	\[
	\norm{T f}_q \leq M \int_{B_{p'}} |l(f)| \, d\nu(l), \qquad \forall f \in L^p(\mu).
	\]
	
	By Riesz representation theorem, for each continuous linear functional $l \in (L^p(\mu))^*$, there exists $g_l \in L^{p'}(\mu)$ such that
	\[
	l(f) = \int_X f \cdot g_l \, d\mu, \qquad f \in L^p(\mu).
	\]
	
	Applying Hölder's inequality yields
	\begin{align*}
		\norm{T f}_q &\leq M \int_{B_{p'}} |l(f)| \, d\nu(l) \\
		&= M \int_{B_{p'}} \int_X |f . g_l| \, d\mu \, d\nu(g_l) \\
		&\leq M \int_{B_{p'}} \norm{g_l}_{p'} \norm{f}_p \, d\nu(g_l).
	\end{align*}
	
	Taking the sequence of disjoint atoms $\{A_i\}_{i=1}^{\infty}$,we define for each $i \in \mathbb N$ 
	\[
	f_i = \frac{\overline u |u|^{{p'}-2} (E(|w|^q))^{({p'}-1)/q}}{\norm{T}^{{p'}/p} (\mu(A_i))^{(r-{p'})/pr}} \chi_{A_i}.
	\]
	One readly checks that $f_i \in L^p(\Sigma)$ and $\norm{f_i}_p \leq 1$. Hence we have
	\[
	\norm{T f_i}_q^q = \frac{1}{\norm{T}^{p'q/p}} \int_{A_i} \frac{1}{\mu(A_i)^{{(r-p'q)}/r}} \left( (E(|u|^{p'})(A_i))^{1/{p'}} (E(|w|^q)(A_i))^{1/q} \right)^{p'q} \, d\mu.
	\]
	Therefore,
	\[
	\norm{T f_i}_q = \frac{1}{\norm{T}^{{p'}/p}} \left( (E(|u|^{p'})(A_i))^{1/{p'}} (E(|w|^q)(A_i))^{1/q} (\mu(A_i))^{1/r} \right)^{p'}.
	\]
	
	Combining this identity with the Pietsch domination bound established earlier, we obtain
	\begin{align*}
		\norm{T f_i}_q &= \frac{1}{\norm{T}^{{p'}/p}} \left( (E(|u|^{p'})(A_i))^{1/{p'}} (E(|w|^q)(A_i))^{1/q} (\mu(A_i))^{1/r} \right)^{p'} \\
		&\leq M \int_{B_{p'}} \norm{g_l\cdot\chi_{A_i}}_{p'} \norm{f_i}_p \, d\nu(g_l) \\
		&\leq M \int_{B_{p'}} \norm{g_l\cdot \chi_{A_i}}_{p'} \, d\nu(g_l).
	\end{align*}
	
	And so
	\[
	(E(|u|^{p'})(A_i))^{1/{p'}} (E(|w|^q)(A_i))^{1/q} (\mu(A_i))^{1/r} \leq M^{1/{p'}} \norm{T}^{1/p} \left( \int_{B_{p'}} \norm{g_l\cdot \chi_{A_i}}_{p'} \, d\nu(g_l) \right)^{1/{p'}}.
	\]
	
	Since the atoms $\{A_i\}_{i=1}^{\infty}$ are disjoint, we have the norm decomposition $\norm{g_l}_{p'} = \sum_{i=1}^{\infty} \norm{g_l\cdot \chi_{A_i}}_{p'} $. Therefore
	\begin{align*}
		\sum_{i=1}^{\infty} & (E(|u|^{p'})(A_i))^{1/{p'}} (E(|w|^q)(A_i))^{1/q} (\mu(A_i))^{1/r} \\
		&\leq M^{1/{p'}} \norm{T}^{1/p} \sum_{i=1}^{\infty} \left( \int_{B_{p'}} \norm{g_l\cdot \chi_{A_i}}_{p'} \, d\nu(g_l) \right)^{1/{p'}} \\
		&\leq M^{1/{p'}} \norm{T}^{1/p} \left( \int_{B_{p'}} \sum_{i=1}^{\infty} \norm{g_l\cdot \chi_{A_i}}_{p'} \, d\nu(g_l) \right)^{1/{p'}} \\
		&= M^{1/{p'}} \norm{T}^{1/p} \left( \int_{B_{p'}} \norm{g_l}_{p'} \, d\nu(g_l) \right)^{1/{p'}} \\
		&\leq M^{1/{p'}} \norm{T}^{1/p}.
	\end{align*}
	
	Consequently,
	\[
	\sum_{i=1}^{\infty} (E(|u|^{p'})(A_i))^{1/{p'}} (E(|w|^q)(A_i))^{1/q} (\mu(A_i))^{1/r} \leq M^{1/{p'}} \norm{T}^{1/p} < \infty.
	\]
	Which completes the proof.
\end{proof}
\begin{thm}\label{thm2.8}
	Let $1 \leq p<q < \infty$ and let \(p^{\prime}\) and \(q^{\prime}\) be conjugate components to \(p\) and \(q\) respectively. Let $\A \subseteq \Sigma$ be a $\sigma$-finite subalgebra such that the measure space $(X, \A, \mu|_{\A})$ is purely atomic. Then the bounded WCE operator $T = M_w E M_u : L^p(\Sigma) \to L^q(\Sigma)$ is nuclear if and only if
	\[
	\sum_{i=1}^{\infty} (E(|u|^{p'})(A_i))^{1/{p'}} (E(|w|^q)(A_i))^{1/q} (\mu(A_i))^{-1/r} < \infty,
	\]
	where $\{A_i\}_{i=1}^{\infty}$ is a collection of pairwise disjoint $\A$-atoms satisfying $X = \bigcup_{i=1}^{\infty} A_i$, and the exponent $r$ given by $\frac{1}{p} + \frac{1}{r} = \frac{1}{q}$.
\end{thm}

\begin{proof}
	Similar to the proof of Teorem 2.4, for every $f \in L^p(\Sigma)$ we have
	\begin{align*}
		T f &= w \cdot E(uf) = w \cdot \sum_{i=1}^{\infty} E(uf)(A_i) \cdot \chi_{A_i} \\
		&= w \cdot \sum_{i=1}^{\infty} \left( \frac{1}{\mu(A_i)} \int_{A_i} E(uf) \, d\mu \right) \cdot \chi_{A_i} \\
		&= w \cdot \sum_{i=1}^{\infty} \left( \frac{1}{\mu(A_i)} \int_{A_i} u f \, d\mu \right) \cdot \chi_{A_i}.
	\end{align*}
	
	Assume that
	\[
	 \sum_{i=1}^{\infty} (E(|u|^{p'})(A_i))^{1/{p'}} (E(|w|^q)(A_i))^{1/q} (\mu(A_i))^{-1/r} < \infty.
	\]
	
	Define for each $i \in \mathbb N$ 
\[
\phi_i(f) = \int_X (u \cdot \chi_{A_i}) f \, d\mu , g_i = \frac{w \cdot \chi_{A_i}}{\mu(A_i)}.
\]
 Then $u \cdot \chi_{A_i} \in L^{p'}(\Sigma)$ and $g_i \in L^q(\Sigma)$. Clearly, each $\phi_i$ is a bounded linear functional on $L^p(\Sigma)$. More over one can verify that \[
 \norm{g_i}_q = (E(|w|^q)(A_i))^{1/q}(\mu(A_i))^{({1/q})-1}
 \] and
 \[
 \norm{\phi_i}_{p'} = (E(|u|^{p'})(A_i))^{1/{p'}}.
 \] In which $\frac{1}{p} + \frac{1}{p'} = {1} $ (we assume that $\frac{1}{\infty} = {0} $ ) . Therefore,
	\[
	T = M_w E M_u = \sum_{i=1}^{\infty} \phi_i \otimes g_i,
	\]
	and
	\[
	\sum_{i=1}^{\infty} \norm{\phi_i}_{p'} \norm{g_i}_{q} = \sum_{i=1}^{\infty} (E(|u|^{p'})(A_i))^{1/{p'}} (E(|w|^q)(A_i))^{1/q} (\mu(A_i))^{-1/r} < \infty.
	\]
	This implies that $T$ is nuclear.
	
	Conversely, let the WCE operator $T = M_w E M_u$ be nuclear. Then exactly same as the proof of theorem 2.4, there is a Borel probability measure $\nu$ on the unit ball $B_{p'}$ of $L^{p'}(\mu) \simeq (L^p(\mu))^*$ (case \(1< p< \infty\) ),$B_{\infty}$ of $L^{\infty}(\mu) \simeq (L^{1}(\mu))^*$ (case p=1)  such that for some $M > 0$,
	\[
	\norm{T f}_q \leq M \int_{B_{p'}} |l(f)| \, d\nu(l), \qquad \forall f \in L^p(\mu).
	\]
	and
	\[
	\norm{T f}_q \leq M \int_{B_{\infty}} |l(f)| \, d\nu(l), \qquad \forall f \in L^{1}(\mu).
	\]
	By Riesz representation theorem,  for each continuous linear functional $l \in (L^p(\mu))^*$, there exists $g_l \in L^{p'}(\mu)$ such that
	\[
	l(f) = \int_X f \cdot g_l \, d\mu, \qquad f \in L^p(\mu).
	\]
	
		Applying Hölder's inequality yields
	\begin{align*}
		\norm{T f}_q &\leq M \int_{B_{p'}} |l(f)| \, d\nu(l) \\
		&= M \int_{B_{p'}} \int_X |f . g_l| \, d\mu \, d\nu(g_l) \\
		&\leq M \int_{B_{p'}} \norm{g_l}_{p'} \norm{f}_p \, d\nu(g_l).
	\end{align*}
	
	Taking the sequence of disjoint atoms $\{A_i\}_{i=1}^{\infty}$,we define for each $i \in \mathbb N$
	\[
	f_i = \frac{\overline u |u|^{{p'}-2} (E(|w|^q))^{({p'}-1)/q}}{\norm{T}^{{p'}/p} (\mu(A_i))^{({p'}+r)/pr}} \chi_{A_i}.
	\]
	One readly checks that $f_i \in L^p(\Sigma)$ and $\norm{f_i}_p \leq 1$. Hence we have
	\[
	\norm{T f_i}_q^q = \frac{1}{\norm{T}^{p'q/p}} \int_{A_i} \frac{1}{\mu(A_i)^{({r+p'q})/r}} \left( (E(|u|^{p'})(A_i))^{1/{p'}} (E(|w|^q)(A_i))^{1/q} \right)^{p'q} \, d\mu.
	\]
	Therefore,
	\[
	\norm{T f_i}_q = \frac{1}{\norm{T}^{{p'}/p}} \left( (E(|u|^{p'})(A_i))^{1/{p'}} (E(|w|^q)(A_i))^{1/q} (\mu(A_i))^{-1/r} \right)^{p'}.
	\]
	
	Combining this identity with the Pietsch domination bound established earlier, we obtain
	\begin{align*}
		\norm{T f_i}_q &= \frac{1}{\norm{T}^{{p'}/p}} \left( (E(|u|^{p'})(A_i))^{1/{p'}} (E(|w|^q)(A_i))^{1/q} (\mu(A_i))^{-1/r} \right)^{p'} \\
		&\leq M \int_{B_{p'}} \norm{g_l\cdot \chi_{A_i}}_{p'} \norm{f_i}_p \, d\nu(g_l) \\
		&\leq M \int_{B_{p'}} \norm{g_l\cdot \chi_{A_i}}_{p'} \, d\nu(g_l).
	\end{align*}
	
	And so
	\[
	(E(|u|^{p'})(A_i))^{1/{p'}} (E(|w|^q)(A_i))^{1/q} (\mu(A_i))^{-1/r} \leq M^{1/{p'}} \norm{T}^{1/p} \left( \int_{B_{p'}} \norm{g_l \cdot\chi_{A_i}}_{p'} \, d\nu(g_l) \right)^{1/{p'}}.
	\]
	
	It is clear that $\norm{g_l}_{p'} = \sum_{i=1}^{\infty} \norm{g_l\cdot \chi_{A_i}}_{p'} $ and $\norm{g_l}_{\infty} = \sum_{i=1}^{\infty} \norm{g_l\cdot \chi_{A_i}}_{\infty} $. Therefore
	\begin{align*}
		\sum_{i=1}^{\infty} & (E(|u|^{p'})(A_i))^{1/{p'}} (E(|w|^q)(A_i))^{1/q} (\mu(A_i))^{-1/r} \\
		&\leq M^{1/{p'}} \norm{T}^{1/p} \sum_{i=1}^{\infty} \left( \int_{B_{p'}} \norm{g_l\cdot \chi_{A_i}}_{p'} \, d\nu(g_l) \right)^{1/{p'}} \\
		&\leq M^{1/{p'}} \norm{T}^{1/p} \left( \int_{B_{p'}} \sum_{i=1}^{\infty} \norm{g_l\cdot\chi_{A_i}}_{p'} \, d\nu(g_l) \right)^{1/{p'}} \\
		&= M^{1/{p'}} \norm{T}^{1/p} \left( \int_{B_{p'}} \norm{g_l}_{p'} \, d\nu(g_l) \right)^{1/{p'}} \\
		&\leq M^{1/{p'}} \norm{T}^{1/p}.
	\end{align*}
	
	This implies that
	\[
	\sum_{i=1}^{\infty} (E(|u|^{p'})(A_i))^{1/{p'}} (E(|w|^q)(A_i))^{1/q} (\mu(A_i))^{-1/r} \leq M^{1/{p'}} \norm{T}^{1/p} < \infty.
	\]
	Which completes the proof.
\end{proof}

\begin{thm}\label{thm2.9}
	Let $\A \subseteq \Sigma$ be a $\sigma$-finite subalgebra. Then the bounded WCE operator $T = M_w E M_u : L^p(\Sigma) \to L^q(\Sigma)$ is nuclear (where  $1 \leq q<p < \infty$) if and only if
	\begin{itemize}
		\item[(i)] $(E|u|^{p^{\prime}})^{\frac{1}{p^{\prime}}}(E|w|^{q})^{\frac{1}{q}} = 0$ $\mu$-a.e. on $B$, and
		\item[(ii)]
		 $\sum_{n=1}^{\infty} (E(|u|^{p'})(A_i))^{1/{p'}} (E(|w|^q)(A_i))^{1/q} (\mu(A_i))^{1/r} < \infty$,
	\end{itemize}
	in which $\{A_i\}_{i=1}^{\infty}$ are pairwise disjoint $\A$-atoms and $B$ is non-atomic such that $X = \bigcup_{i=1}^{\infty} A_i \cup B$ and the exponent r determined by $\frac{1}{p} + \frac{1}{r} = \frac{1}{q}$.
\end{thm}

\begin{proof}
	The statement follows directly by combining Theorem \ref{thm2.7} and Corollary \ref{cor2.5}.
\end{proof}

\begin{thm}\label{thm2.10}
	Let $\A \subseteq \Sigma$ be a $\sigma$-finite subalgebra. Then the bounded WCE operator $T = M_w E M_u : L^p(\Sigma) \to L^q(\Sigma)$ is nuclear (where  $1 \leq p<q < \infty$) if and only if
	\begin{itemize}
		\item[(i)] $(E(|u|^{p^{\prime}})(E(|w|^{q}))^{\frac{p^{\prime}}{q}} = 0$ $\mu$-a.e. on $B$, and
		\item[(ii)]
		 $\sum_{n=1}^{\infty} (E(|u|^{p'})(A_i))^{1/{p'}} (E(|w|^q)(A_i))^{1/q} (\mu(A_i))^{-1/r} < \infty$,
	\end{itemize}
	in which $\{A_i\}_{i=1}^{\infty}$ are pairwise disjoint $\A$-atoms and $B$ is non-atomic such that $X = \bigcup_{i=1}^{\infty} A_i \cup B$ and the exponent r determined by $\frac{1}{q} + \frac{1}{r} = \frac{1}{p}$.
\end{thm}
\begin{proof}
		The statement follows directly by combining Theorem \ref{thm2.8} and Corollary \ref{cor2.5}.
\end{proof}
\begin{exam}
	Consider the measure space $(\mathbb{N}, 2^{\mathbb{N}}, \mu)$ where $\mu$ is the counting measure.  
	A sub‑$\sigma$-algebra $\mathcal{A}$ is generated by the following partition of $\mathbb{N}$:
	\begin{itemize}
		\item \textbf{Odd atoms:} $\{1\}, \{3\}, \{5\}, \dots$ each of measure $1$;
		\item \textbf{Even atoms:}  
		$A_1 = \{2\}$, $A_2 = \{4,6\}$, $A_3 = \{8,10,12\}$, $A_4 = \{14,16,18,20\}$, $\dots$  
		where $\mu(A_n)=n$ and each $A_n$ is an $\mathcal{A}$-atom consisting of $n$ consecutive even numbers.  
		The first element of $A_n$ is $2k_n$ with $k_n = \frac{n(n-1)}{2}+1$ (hence $k_n \ge n$).
	\end{itemize}
	There is no non‑atomic part.\\
	The conditional expectation operator $E = E^{\mathcal{A}}$ acts as follows:  
	
	\begin{itemize}
		\item On an odd atom $\{2k-1\}$: $E(f)(2k-1)=f(2k-1)$;
		\item On an even atom $A_n$: $E(f)$ is constant on $A_n$ and equals  
		$\displaystyle \frac{1}{n}\sum_{j=0}^{n-1} f\bigl(2(k_n+j)\bigr)$.
	\end{itemize}
	Two weight sequences are introduced:  
	\begin{align*}
		u(n)=n,\qquad w(n)=\frac{1}{n^{3}}\qquad (n\in\mathbb{N}).
	\end{align*}
	The WCE operator $T$ is defined by  
	\begin{align*}
	Tf = w\cdot E(uf),\qquad f\in L^{p}(\Sigma),
	\end{align*}
	where $\Sigma = 2^{\mathbb{N}}$.
	
	Assume $1\le p<q<\infty$. Set $p' = p/(p-1)$ and let $r$ be determined by $\frac1q+\frac1r=\frac1p$ (i.e. $r=\frac{pq}{q-p}$).  
	In our setting $B=\varnothing$, so part (i) of Theorem \ref{thm2.10} is automatic.  
	We now verify (ii).\\
	\textbf{Odd atoms:}
	  Since
	  \begin{align*}
	  	E(|u|^{p'})(\{2k-1\}) = (2k-1)^{p'},\qquad E(|w|^q)(\{2k-1\}) = (2k-1)^{-3q}, 
	  \end{align*}
	   therfore we have:  
	\begin{align*}
		(E(|u|^{p'})^{1/p'} = 2k-1,\qquad 
		(E(|w|^q))^{1/q} = (2k-1)^{-3}.
	\end{align*} 
	The measure factor $\mu(\{2k-1\})^{-1/r}=1$. Hence the contribution is $(2k-1)^{-2}$.  
	The series $\sum_{k=1}^{\infty} (2k-1)^{-2}$ converges.\\
	\textbf{Even atoms $A_n$ ($\mu(A_n)=n$):}
	  We only need an upper bound. Since all indices in $A_n$ are at least $2k_n$ and $k_n\ge n$,
	\begin{align*}
		(E(|u|^{p'})(A_n))^{1/p'} \le \max_{j=0}^{n-1}2(k_n+j) \le 2(k_n+n)\le 4k_n,
	\end{align*}
	\begin{align*}
		(E(|w|^q)(A_n))^{1/q} \le ((2k_n)^{-3q})^{1/q} = (2k_n)^{-3}.
	\end{align*} 
	Thus  
	\begin{align*}
		(E(|u|^{p'}))^{1/p'}(E(|w|^q))^{1/q} \le 4k_n\cdot(2k_n)^{-3}= \frac{1}{2k_n^{2}}.
	\end{align*}
	Moreover $\mu(A_n)^{-1/r}=n^{-1/r}$ and $k_n\sim n^{2}/2$. Consequently the term for $A_n$ is bounded by a constant times $n^{-4}\cdot n^{-(q-p)/(pq)} = n^{-4-\frac{q-p}{pq}}$, which is summable because the exponent is $<-4$.  
	Hence the whole series converges. Theorem 2.10 immediately yields that \textbf{$T = M_w E M_u : L^p(\Sigma) \to L^q(\Sigma)$ is nuclear for every $1\le p<q<\infty$.}
	
	Now assume $1\le q<p<\infty$. Again set $p' = p/(p-1)$ and define $r$ by $\frac1p+\frac1r=\frac1q$ (so $r=\frac{pq}{p-q}$).  
	Theorem \ref{thm2.9} states that a bounded operator $T : L^p(\Sigma) \to L^q(\Sigma)$ is nuclear iff  
	\begin{itemize}
		\item[(i)] $(E|u|^{p'})^{1/p'}(E|w|^{q})^{1/q}=0$ on $B$;
		\item[(ii)] $\sum_{i=1}^{\infty} (E(|u|^{p'})(A_i))^{1/p'}\,
		(E(|w|^q)(A_i))^{1/q}\,
		(\mu(A_i))^{1/r}<\infty$.
	\end{itemize}
	Again (i) is trivial because $B=\varnothing$.  
	We check (ii).\\
	\textbf{Odd atoms:} The same computation as before gives $(2k-1)^{-2}$ multiplied by $\mu(\{2k-1\})^{1/r}=1$. The series converges.\\
	\textbf{Even atoms $A_n$:} The estimate for the product $(E(|u|^{p'}))^{1/p'}(E(|w|^q))^{1/q}$ is still $\le \frac{1}{2k_n^{2}}$.  
	Now $\mu(A_n)^{1/r}=n^{1/r}=n^{(p-q)/(pq)}$.  
	Using $k_n\sim n^{2}/2$, the term is bounded by a constant times $n^{-4}\cdot n^{\frac{p-q}{pq}}=n^{-4+\frac{p-q}{pq}}$.  
	Since $\frac{p-q}{pq} = \frac1q-\frac1p < \frac1q \le 1$, the exponent satisfies $-4+\frac{p-q}{pq}\le -3$. Hence the series converges absolutely.  
	Therefore Theorem 2.9 applies and \textbf{$T = M_w E M_u : L^p(\Sigma) \to L^q(\Sigma)$ is nuclear for all $1\le q<p<\infty$.}
	
\end{exam}
{\bf Declarations}
\begin{itemize}
\item Conflicts of Interest:
The authors declare no conflicts of interest.
\item Author Contributions:
The authors have the same contribution.
\item Data Availability Statement:
Data is not applicable. The results are obtained by manual computations.
\item Funding:
No funding.
\end{itemize}


\begin{thebibliography}{99}


\bibitem{dou} R.G. Douglas, \emph{Contractive projections on an $L^1$- space}, Pacific J. Math. 15 (1965) 443–462.

\bibitem{Dodds1990}
P.~G. Dodds, C.~B. Huijsmans and B.~De Pagter,
\emph{Characterizations of conditional expectation-type operators},
Pacific J. Math. \textbf{141}(1) (1990), 55--77.

\bibitem{Diestel1995}
J.~Diestel, H.~Jarchow and A.~Tonge,
\emph{Absolutely summing operators},
Cambridge Studies in Advanced Mathematics, vol.~43, Cambridge University Press, Cambridge, 1995.

\bibitem{Emamalipour2020}
H.~Emamalipour and M.~R. Jabbarzadeh,
\emph{Lambert Conditional Operators on $L^2(\Sigma)$},
Complex Anal. Oper. Theory \textbf{14}, 18 (2020).

\bibitem{Estaremi2015}
Y.~Estaremi,
\emph{Multiplication conditional expectation type operators on Orlicz spaces},
Journal of Mathematical Analysis and Applications \textbf{414} (2015), 88--98.

\bibitem{Estaremi2015b}
Y.~Estaremi,
\emph{On a Class of Operators With Normal Aluthge Transformations},
Filomat \textbf{29} (2015), 969--975.

\bibitem{Estaremi2018}
Y.~Estaremi,
\emph{On the algebra of WCE operators},
Rocky Mountain J. Math. \textbf{48} (2018), 501--517.

\bibitem{Estaremi2013}
Y.~Estaremi and M.~R. Jabbarzadeh,
\emph{Weighted Lambert type operators on $L^p$-spaces},
Oper. Matrices \textbf{1} (2013), 101--116.

\bibitem{Estaremi2014}
Y.~Estaremi and M.~R. Jabbarzadeh,
\emph{Compact Lambert type operators between two $L^p$ spaces},
J. Math. Anal. Appl. \textbf{420} (2014), 118--123.

\bibitem{gdp} J. Grobler, B. de Pagter, \emph{Operators representable as multiplication-conditional expectation operators}, J. Operator Theory
48 (2002) 15–40.

\bibitem{gdp1} J. Grobler, B. de Pagter, \emph{Operators represented by conditional expectations and random measures}, Positivity 9 (2005)
369–383. 

\bibitem{Herron2011}
J.~Herron,
\emph{Weighted conditional expectation operators},
Oper. Matrices \textbf{1} (2011), 107--118.

\bibitem{lam} A. Lambert, \emph{$L^p$-multipliers and nested sigma-algebras}, Oper. Theory Adv. Appl. 104 (1998) 147–153.

\bibitem{moy} Shu-Teh Chen Moy, \emph{Characterizations of conditional expectation as a transformation on function spaces}, Pacific J. Math.
4 (1954) 47–63.

\bibitem{Pietsch1972}
A.~Pietsch,
\emph{Nuclear Locally Convex Spaces},
Ergebnisse der Mathematik und ihrer Grenzgebiete, Band~66, Springer-Verlag, New York-Heidelberg, 1972.

\bibitem{Rao1993}
M.~M. Rao,
\emph{Conditional measure and applications},
Marcel Dekker, New York, 1993.

\bibitem{Reinov2001}
O.~I. Reinov,
\emph{Approximation Properties and Some Classes of Operators},
Journal of Mathematical Sciences \textbf{107} (2001), 3911--3951.

\bibitem{Reinov2014}
O.~I. Reinov,
\emph{On linear operators with $s$-nuclear adjoints, $0 < s \leq 1$},
J. Math. Anal. Appl. \textbf{415} (2014), 816--824.

\bibitem{Zaanen1967}
A.~C. Zaanen,
\emph{Integration},
2nd ed., North-Holland, Amsterdam, 1967.
\end{thebibliography}
\end{document}